\documentclass[11pt]{amsart}
\usepackage{geometry}                
\usepackage[parfill]{parskip}    
\usepackage{graphicx}
\usepackage{amsthm}
\usepackage{amsmath}
\usepackage{amssymb}
\usepackage{epstopdf}
\usepackage{enumerate}
\usepackage{graphicx}
\usepackage{hyperref}
\usepackage[utf8]{inputenc}

\title[A new trick for Santalo's proof of the strong Bonnesen inequality]
{A new trick for Santalo's integral geometric proof of the strong Bonnesen inequality}

\author{Michael E. Gage}
\thanks{I wish to thank John Harper for encouragement and support. Much of the material in this article is based on the Gehman Lecture I delivered to the Seaway Section meeting of the MAA in April 2008}
\date{\today}                                           
\dedicatory{In memory of Robert Osserman and Eugenio Calabi}

\newcommand{\p}{\partial}

\newcommand{\la}{\left <}
\newcommand{\ra}{\right >}
\newcommand{\rin}{r_{\text{in}}}
\newcommand{\rout}{r_{\text{out}}}
\newcommand{\rhoin}{\rho_{\text{in}}}
\newcommand{\rhoout}{\rho_{\text{out}}}
\theoremstyle{plain}
\newtheorem{theorem}{Theorem}

\newtheorem{lemma}{Lemma}

\theoremstyle{definition}
\newtheorem{definition}{Definition}

\theoremstyle{remark}

\begin{document}
{ current version \today}

\maketitle
\section{Introduction}
The goal in this  paper is to give a review of the isoperimetric-type inequalities surrounding the proof that the isoperimetric ratio $$\frac{L^{2}}{A}\label{E:isoperimetricratio}$$ decreases to its minimum $4\pi$ under the curve shortening flow 
\begin{equation}
X_{t}=kN\label{E:curveshortening}
\end{equation} and  to make one novel improvement to the proof of the Santalo-Bonnesen theorem.  (see \cite{Gage1983} or {\cite{Epstein-Gage1987} for background on the curve shortening problem.)
\begin{itemize}
\item We define the Bonnesen functional $B(r)$ and emphasize  its importance in deriving other isoperimetric  inequalities  including an integral bound on the support function $h$ 
\begin{equation}
\int_{\p K }h^{2}ds \le \frac{LA}{\pi}
\end{equation}
which is  the key formula used to prove that the  curvature flow decreases the isoperimetric ratio.  

\item We  quickly review the Santalo proof of the Bonnesen inequality via integral geometry  which places a lower bound on the isoperimetric deficit 	
$(L^{2}-4\pi A)\ge \pi^{2}(\rout-\rin)^2$	
\item We define the concept of a positive center, first introduced in \cite{Gage1990}
\item We then recall the minimal annulus theorem for a convex closed curve. 
\item We review the earlier proofs   that the center of the minimal annulus is a positive center using  "cut and flip" symmetrization techniques.
\item We  give a detailed proof of the ``averaging trick" which allows us to  directly extend the Santalo integral geometric argument to prove that $B(r)\ge 0$ on the interval $[\rho_{in},\rho_{out}]$ avoiding the symmetrization techniques.
I have not seen this ``averaging trick"   used previously in the context of the Santalo-Bonnesen theorem.  
\end{itemize}

In the second section of this paper we follow essentially the same outline to prove analogs of these results for Minkowski geometries.

\subsection{The centrality of the Bonnesen functional}

\begin{itemize}
\item We define the Bonnesen functional for a smooth closed curve $\p K$ bounding a lamina $K$ in the plane as
$$ B(r):= -\pi r^{2} + Lr-A$$ where $L$ is the perimeter of the curve, $A$ is the area of the lamina  and $r$ is, for now, an arbitrary value.  This formula is sometimes known as the extended Bonnesen inequality since  both the isoperimetric inequality and the Bonnesen inequality can be magically derived by examining the interval on which $B(r)\ge0$.  (See \cite{Osserman1979} for an extended account of the Bonnesen inequality and its relatives.)

\item If $B(r)$ is positive for any  $r$ then the quadratic has real roots and the discriminant $L^{2} - 4\pi A$ is non-negative, which is exactly the isoperimetric inequality.

\item We define the inradius and the outradius of a closed convex curve:  the inradius $\rin$ is the radius of the largest circle contained within the lamina bounded by the curve.  The outradius $\rout$ is the radius of the smallest circle containing the lamina. 

\item We will prove that $B(\rin)\ge 0$ and $B(\rout)\ge0$ therefore because  $B(r)$ quadratic is concave down, $B(r)\ge 0$ for all $r$ in the interval $[\rin, \rout]$. 

\item\label{T:bonnesenInequalityIntro}
This also means that the interval  $[\rin, \rout]$ is trapped between the roots of $B$. These roots can be determined using the quadratic formula hence $\rin\geq \frac{L}{2\pi}-\frac{\sqrt{L^{2}-4\pi A}}{2\pi}$ and $\rout\leq \frac{L}{2\pi}+\frac{\sqrt{L^{2}-4\pi A}}{2\pi}$. 
which after subtracting and simplifying implies the standard  \textbf{\ Bonnesen inequality} for the isoperimetric deficit $$L^{2}-4\pi A \geq \pi^{2} (\rout -\rin)^{2}  $$  This also implies that if the isoperimetric deficit is zero   then the curve is a circle.  (See Theorem \ref{T:BonnesenInequality}) for the justification that $B(\rin)$ and $B(\rout)$ are non-negative.)

\item 
Using the minimal width annulus theorem and ``averaging trick"  in this paper we can show that $B(r)$ is positive on the larger interval  $p\in[\rho_{in}, \rho_{out}]$ where $\rho_{in}$ is the radius of the inner circle of the annulus and $\rho_{out}$ is the outer radius.     This implies  a stronger version of the Bonnesen inequality for the isoperimetric deficit $L^{2}-4\pi A \ge \pi^{2}(\rho_{out} -\rho_{in})^{2}$and also implies that the center of the minimal width annulus is a positive center for the curve $\p K$.  I believe the ``averaging trick"  is the one new result in this paper. Although it is a simple idea I have not seen it used elsewhere in the context of proving the Santalo-Bonnesen theorem.
\end{itemize}

\subsection{Definitions and basic formulas}

The support function $h$ of the curve is defined by $h(s)=<X(s), -N(s)>$ where $X$ is the position vector of the curve and $N$ is the inward pointing  unit normal vector. The vector $X$ and the support function $h$ are based at the origin. The unit tangent vector is denoted $T$. 
The area enclosed by the curve can be derived from the Green's theorem and is given by
\begin{equation}
A=\frac12 \int_{0}^{L} (xy'-yx')ds = \frac12 \int_{0} ^{L} \la X, -N\ra\,ds=\frac12 \int_{0}^{L} h ds
\end{equation}

Using $X'(s)=T(s)$ and $T'(s)=\kappa N(s)$ where $\kappa$ is the curvature  we obtain
\begin{equation}
\int _{\p K} h\kappa ds = -\int_{0}^{L}\la X, \kappa N\ra ds = - \int_{0}^{L} \la X, X''\ra ds = -\la X, X'\ra  \left |_{0}^{L} \right.+\int _{0}^{L} \la X', X'\ra ds = L
\end{equation}

\subsection{the derivative of the isoperimetric ratio} 

A calculation (\cite{Gage1983} and elsewhere) shows that the derivative of the isoperimetric ratio is given by 
\begin{equation}\label{derivativeOfIsoperimetricRatio}
\left(\frac{L^{2}}{A}\right)_{t} = -2 \frac{L}{A}\left(\int_{\p K}\kappa^{2}ds-\pi \frac{L}{A}\right).
\end{equation}
In order to show that this is negative we first observe that
\begin{equation}
L=\int_{\p K} h\kappa ds \le \left( \int_{\p K}h^{2}ds\right)^{1/2}\left(\int _{\p K}\kappa^{2}ds \right)^{1/2},\label{E:intkh}
\end{equation}

If we can show that 
\begin{equation}\label{E:supportsqint}
\int_{\p K}h^{2}ds \leq \frac{LA}{\pi}
\end{equation} 
then substituting that result into the equation \eqref{E:intkh} above and simplifying we obtain that the derivative of the isoperimetric ratio is negative.
\begin{align*}
L^2&\le \int \kappa^{2}\,ds \int h^{2}\,ds\le \frac{LA}{\pi}\int \kappa^{2}\\ 
\pi \frac{L}{A}&\le \int\kappa^{2}\,ds
\end{align*}

\subsection{Positive center\label{S:positivecenter}}
We hope to find a point $\mathcal{O}$ to use as the origin so that the support function $h(s)$ of the curve from that point satisfies 
\begin{equation}
B(h(s)) = - \pi h(s)^{2} +L h(s) -A\geq 0 \text{  for all $s$} 
\end{equation}   
where $s$ is the arc length parameter along the curve.
\begin{definition}The {\bf positive center} of a curve is a point where the support function based at that point satisfies $B(h(s)))\ge 0$ for all points on the curve
\end{definition}
Basing the support function at a positive center  and integrating $B(h)$ over the curve yields
\[
0\leq \int_{\p K} -\pi h^{2} + Lh -A ds = -\pi \int h^{2}ds +2LA -AL
\]
which simplifies to 
\begin{equation}
\int h^{2}ds \leq \frac{LA}{\pi}
\end{equation}

 It turns out that this inequality has a number of consequences and in particular it is the basis of the  proof that the isoperimetric ratio decreases under the curve shortening flow. 
 
 It is easy to find a positive center for a curve $\p K$ which is symmetric through the origin since the origin itself is the positive center and the inradius circle and the outradius circle centered at the origin define the minimal annulus which surrounds the curve.  For any curve in that annulus $\rin \leq h(s) \leq \rout $  the original Santalo-Bonnesen proof shows that $B(h(s))\geq 0$ for all $s \in \p K$
 
 If the curve is not symmetric through the origin there are several ways  to proceed.  
 \begin{itemize}
 \item
 The first approach [Gage1983] was to use the intermediate value theorem to find a line  which splits the area of the lamina in half and for which the tangents where the line segment meets the curve are anti-parallel. The line defines two figures, each in a half plane, and by considering them separately, doubling each figure by flipping the figures around the center of the line segment we obtain two origin symmetric convex figures with $C^{1}$ boundary for which
 \begin{equation}
 \int_{left} h^{2}\leq \frac{L_{left}A}{\pi}\qquad  \int_{right} h^{2}\leq \frac{L_{right}A}{\pi}
 \end{equation}
 These inequalities can then be added to prove \eqref{E:supportsqint} for the original curve. This proves that the isoperimetric ratio for curves symmetric through the origin is non-increasing. This method doesn't necessarily identify a positive center since the inradius and outradius of the left and right bodies might be different.
 
 \item A more detailed analysis in [Gage1990]  of this procedure  shows that the center of the minimal width annulus enclosing the lamina is a positive center  using a more elaborate cut, flip and paste technique.    Specifically if $\mathcal{O}$ is the center and $\rho_{in}$ is the radius of the inner boundary of the annulus, $\rho_{out}$ is the radius of the outer boundary, then $B(\rho_{in})$ and $B(\rho_{out})$ are both positive and therefore, by convexity of  $B$, $B(r)\ge0$ on $[\rho_{in}, \rho_{out}]$, i.e. $B(h(s))\ge0$ for all $s$ in the curve $\p K$ where $h$ is the support function based at the positive center hence inequality \eqref{E:supportsqint}  holds. 
 Inequality \eqref{E:supportsqint} and \eqref{derivativeOfIsoperimetricRatio} imply that the isoperimetric ratio is non-increasing. More is true: using \cite{Gage1984} and \cite{Angenent1991}one can show that the solution to \eqref{E:curveshortening} exists  on a time interval $0\le t< T$ with $\lim_{t\to T} A=0$. and  that the isoperimetric ratio decreases to its minimum $4\pi$.  Using Bonnesen's original inequality this proves that the curve, when normalized to have unit area, converges to the boundary of the unit disk. 
 
The positivity of $B$ also implies a strong Bonnesen inequality  
$$L^2-4\pi A \ge \pi^2(\rho_{out} -\rho_{in})^2$$
 an inequality, which Bonnesen \cite{Bonnesen1921}  proved using angular symmetrization.

\cite{Gage1990} also showed that standard ``centers'' such as the center of mass, the Steiner point, the centroid of the boundary and the minimizer of $\int h^{2}ds$ are not necessarily positive centers.
 
 \item An extension of this procedure to Minkowski space (where the unit disk is replaced by any strictly convex set symmetric through the origin) is described in [Gage1993] \cite{Peri-Wills-Zucco1993}and elsewhere. 
 \end{itemize}
  
\subsection{Extended  Santalo-Bonnesen integral geometry proof of the isoperimetric inequality}

 There is a beautiful proof of the isoperimetric inequality in the plane using integral geometry which appears in Santalo's book \cite{Santalo1953,Santalo1967} and can now also be found in many other places. 

Integral geometry, or geometric probability, defines a canonical measure on the space of lines, or the space of bodies in the plane.  Averaging intersection numbers of a body chosen at random according to this measure and a fixed body yields geometric properties of the two bodies.  

\subsection{Integral geometric measure}
\begin{itemize}
\item Integral geometric measure for lines: Define a line using polar coordinates let $\theta$ be the angle of the shortest distance to the line from the origin and $r$ the distance to the line. Then $drd\theta$ defines a measure on lines which is invariant under Euclidean motions.  
\item Integral geometric measure for solid bodies: Attach an orthonormal 2-frame to the body and define the base point of the frame using  coordinates $x$ and $y$ and the rotation of the frame by $\theta$.  Then the measure $dx dy d\theta$ defines a measure which is invariant under Euclidean motions.
\end{itemize}

\subsection{Crofton's formula for the length of the boundary of a convex lamina in terms of weighted integral geometric measure of its intersection  with a lines in the plane.}
 
 $$2L=\int n dr d\theta$$
 In words: twice the length of the  boundary is equal to the expected  number of boundary intersections with random lines using the integral geometric measure on lines in the plane.   $n$ is the number of intersections of a given random line with the boundary. 
 	
\subsection{Poincare formula:\label{S:poincare}}
$$ 4L_{1}L_{2} = \int n dx\,dy\,d\theta$$
In words: If the first convex body is fixed  and a second convex body is positioned at random with respect to integral geometric measure then  the expected number of intersections of their boundaries is 4 times the product of the length of their boundaries.  Notice that in this case the $n$ is an even number (perhaps zero) almost always.  
 
 \subsection{ Blaschke formula:\label{S:blaschke}}
 
 $$\int  \nu dx\,dy\,d\theta= 2\pi (A_{1}+ A_{2}) +L_{1}L_{2}$$
 where $\nu$ is the number of components in the intersection of the interiors of the two bodies.  Again the first body is fixed and the second (moving) body is distributed according to geometric measure. Since we are dealing with convex objects in this paper $\nu$ is always $1$ or $0$. 
 
 We will use most often the special case where the moving body is a Euclidean disk (or later the centrally symmetric unit disk of a Minkowski geometry.  We first consider a disk whose radius $r$ is between $r_{in}$ the inradius, or radius of the largest disk contained in the stationary  lamina surrounded by the curve $\p K$ and $r_{out}$, the outradius of the stationary body.  Then it is clear geometrically for $r$ in this interval that $n-2\nu \ge 0$ for almost all positions since  if the disk touches the convex lamina then it must intersect the boundary curve, generically in an even number of points.  Tangential touches have lower dimensionality and are of measure zero for integral geometric measure.  

 \subsection{Determining the interval on which the Bonnesen function is positive}
 Subtracting the Blaschke formula from the Poincare formula and assuming that $A_{2}$ is the moving Euclidean disk's area  and $A_{1}$ is the area of the lamina $K$ interior to the curve $\p K$ implies that 
 $$ \int d\theta \int (n -2\nu) dx dy  = -4\pi A_{1}-4\pi A_{2}+2L_{1}L_{2}$$ 

 The symmetry of the Euclidean disk means that the rotational motion of the geometric measure doesn't change the integrand and that $d\theta$ can be factored out.  
 For a disk of radius $r$ we have $A_{2}= \pi r^{2}$ and $L_{2}=2\pi r$. Substituting all this in the formula above and dividing by $4\pi$ to simplify the coefficients we obtain the defining formula for  the (Euclidean) Bonnesen functional $B(r)$ for a convex closed curve $\p K$  bounding a lamina $K$
\begin{equation}
B(r)=\int \left(\frac{n}{2} - \nu \right) dx dy = rL_{1}-A_{1}-\pi r^{2}\tag{Bonnesen functional}
\end{equation}
\begin{theorem}[The Bonnesen inequality]
\label {T:BonnesenInequality}
$B(r)\ge 0$ for all $r\in [\rin,\rout]$

As explained in the introduction \ref{T:bonnesenInequalityIntro} it follows easily from this that $L^{2}-4\pi A \ge \pi^{2}(\rout - \rin)^{2}$.
\end{theorem}

\begin{proof}
 If the moving disk has radius $r$ with $r\in[\rin,\rout]$ then the disk can neither contain the lamina bounded by the curve nor be contained in that lamina.  Hence, except for positions of measure $0$, every time the disk touches the lamina ($\nu=1$) the boundary of the disk must intersect the boundary of the lamina (generically) in two or more places.  This implies that the integrand is always non-negative and that $B(r)\ge0$ for $r\in[\rin,\rout]$.
 This is Santalo's proof that the Bonnesen functional $B(r)$ is greater than or equal to $0$  when $r\in  [r_{in},r_{out}]$ and that equality occurs only for the circle. 
 \end{proof}
 
 Our proof of the strengthened version of this inequality  will show that $B(r)\ge 0$  for $r\in [\rho_{in},\rho_{out}]$, the range between the inner radius of the minimal annulus containing the curve $\p K$ and the outer radious of the annulus, even though the integrand is sometimes negative, and that therefore the center of the minimal annulus is a positive center, that is $B(h(s))\ge0$ for all $s\in\p K$.  
 
 It does this by balancing the measure of the positions where the disk of radius $\rho_{in}$ has intersection numbers $n=0$ and $\nu =1$), (i.e. the disk is completely contained in the fixed body),  with the positions where $n\ge 4$ and $\nu=1$.  This implies that the integral is positive, even though the integrand is sometimes negative, and $B(\rho_{in})\ge 0$.  To the best of my knowledge this balancing approach has not previously been used in this context.
 We then do a similar analysis with the disk of radius $\rho_{out}$ to prove that $B(\rho_{out})\ge0$. Since $B$ is convex downward this completes the proof that $B(h)\ge 0$ on $[\rho_{in}, \rho_{out}]$. 
  
 \subsection{The minimal annulus of a convex  curve (equivalently of its enclosed lamina) in the plane\label{S:minimalannulus}}
 
 For any simple convex closed curve  there exists a unique minimum  width annulus and this annulus will have at least 4 boundary contacts between the inner and outer boundaries of the minimal annulus and the closed curve, with the contact points alternating between the inner and outer boundaries.   This is proved in [Bonnesen 1929].

 \begin{theorem}[Minimal annulus theorem]\label{T:minimalannulus} let $K$ be the lamina interior to the convex curve and let $\mathcal{T}$ be the unit Euclidean disk. We can find a translation and  homotheties of $\mathcal{T}$  so that   
 $$\rho_{in} \mathcal{T} \subseteq K \subseteq \rho_{out} \mathcal{T}$$
  and $\rho_{out}-\rho_{in}$ is minimized.  One can prove that in that  case  there are points $a, A, b, B$ on the boundary of $K$ which alternately touch tangentially $\rho_{in}\mathcal{T} $
 ( at $a$ and $b$) and $\rho_{out}\mathcal{T}$ (at $A$ and $B$).  In addition $\rhoin\le h \le \rhoout$.
 \end{theorem}
\begin{proof}
In  \cite{Gage1993}this theorem is proved by using   approximation theory  applied to the support function. 
It is clear that  the existence of a minimal Euclidean annulus is equivalent to the $L_{0}$ approximation problem of  finding $a$ and $b$ so that the range of 
 $$h(\theta)+ a\cos \theta+ b\sin \theta$$
 is minimized. 
 
 See \cite{Gage1993} Lemma 5.2 for details for completing the proof along with references to G. G. Lorentz's book ``Approximation of Functions'' \cite{Lorentz1966} particularly pages 18-27 on Cheyshev systems of functions. ( Bonnesen's work included a lot of approximation theory and he was probably aware of this connection with geometry.  )
\end{proof}
 At this point our proof diverges from the proof in \cite{Gage1993} which uses cut and paste symmetrization techniques to prove that the center of the minimal annulus is a positive center.  

\newpage

 \begin{figure}[h]
\begin{center}
\includegraphics[scale=0.5]{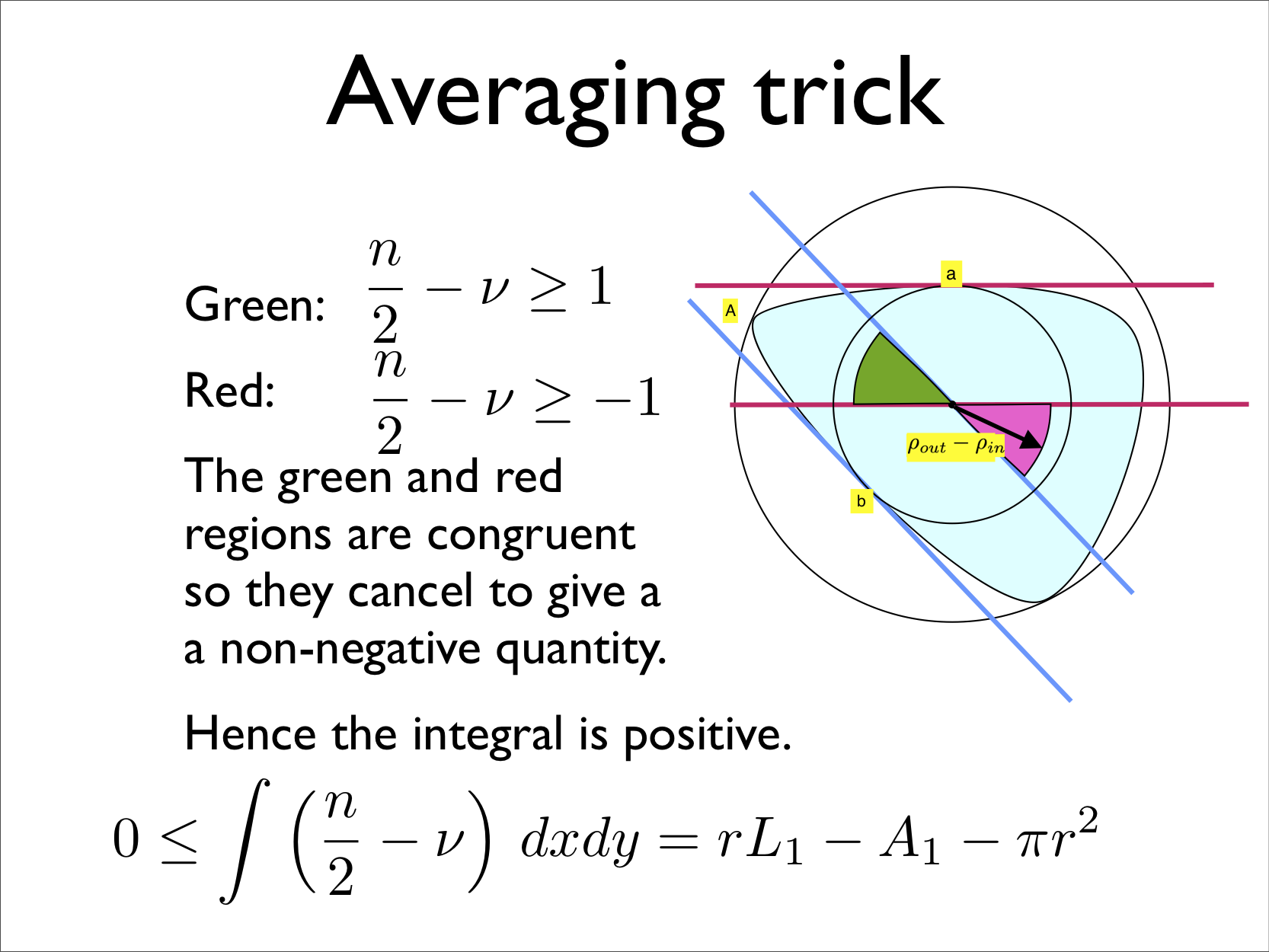}
\caption{averaging trick for the moving circle of radius $\rho_{in}$}
\label{default}
\end{center}
\end{figure}
). 

\begin{theorem}\label{T:poscenter} Using a Santalo integral geometry style proof we show that the Bonnesen functional $B(h)$  is non-negative on the interval $[\rho_{in},\rho_{out}]$.  If the support function $h$ is based at the center of the annulus then $\rhoin\le h(s)\le\rhoout$ for all points $s$ on the curve, $B(h(s))\ge 0$ and the center of the minimal annulus is a positive center. 
\end{theorem}
A proof of the theorem itself can be found in  \cite{Gage1990} which uses cut, flip and paste techniques.  The new twist in this proof is the use of integral geometric techniques to avoid the  cut and paste comparisons.

\subsection{$B(\rhoin)\ge 0$}
We wish to determine explicitly the positions of a moving  disk  of radius $\rho_{in}$ where the integrand $\frac{n}{2}-\nu$ is $-1$, $0$, or greater than $1$. (Since we are dealing with convex bodies the intersection of the interiors $\nu$ is always either $1$ or $0$. The intersection of the boundaries is in general an even number except for sets of measure 0. ) 

 We draw a straight line tangent to the contact point $a$ and a parallel line through the center of the annulus.  The boundary of a random circle of radius $\rho_{in}$ whose center lies in that strip must intersect the tangent through $a$  and since that intersection point lies outside the lamina the boundary of the random inradius circle must intersect the boundary of the lamina in  at least 2 points.  

 Similarly we  draw the tangent at $b$ and draw a line parallel to that through the center of the annulus and observe that $\rho_{in}$ circles that lie in that strip intersect also intersect the boundary of the lamina in at least 2 points.  So for centers which lie in the triangle which is the intersection of both strips, the boundary of the disks intersect the lamina in at least 4 points (which are distinct if the center of the moving disk lies in the green region) and those that lie in neither strip don't intersect either tangent line and therefore might not intersect the boundary of the lamina at all. In this region it is possible that the $\frac{n}{2}-\nu$ integrand is $-1$, i.e. the moving  disk is entirely within the lamina and $n=0$ while $\nu=1$. 
 
 For definiteness we assume that the tangent lines intersect to the left of this configuration.  Notice that the their intersection point must lie outside the outer circle of the annulus since the point $A$ on the boundary of the lamina between $a$ and $b$ lies on the outer circle.  
 
 We also consider positions where the center of the moving disk lies outside the circle of radius $\rho_{out}-\rho_{in}$.  These disks must have points which lie outside the outer boundary of the annulus so they can't be contained within the lamina. If they don't intersect the lamina at all then $n=0$ and $\nu=0$ and the integrand is zero and their contribution to the integral is $0$.  If they do intersect the lamina   then the boundaries of the lamina and the disk intersect and the integrand $\frac{n}{2}-\nu$ is at least $0$ and so they contribute non-negatively to the integral. 
 
 The result is two regions, cut off by the lines through the center of the annulus and the circle of radius $\rho_{out}-\rho_{in}$.  The left region, (in green) is guaranteed to have 4 boundary intersection points, so the the integrand is at least 1 on this region, while the right region has an integrand  $-1$ or greater, but since it is congruent to the green region on the left  the net contribution of the integrand on these two regions to the integral is non-negative.  
 
 More details on why disks centered in the green region must intersect the boundary in 4 points:

 There is one configuration (of measure 0) when two of the boundary intersections coincide for a disk with a center in the green region.  This is when the center of the disk is on boundary of the $\rho_{out}-\rho_{in}$ triangle segment and the edge of the disk on the $\rho_{out}$ boundary is precisely at $A$.  Anywhere else the moving circle  boundary intersects first the boundary of the lamina, then the blue tangent line, then intersects the blue tangent line again, then hitting boundary of the lamina again between $A$ and $b$.  It then goes through a similar arc heading toward $a$, passing through the boundary between $A$ and $a$, then the red tangent line and finally the lamina boundary near $a$.  
 
  \begin{figure}[h]
\begin{center}
\includegraphics[scale=0.5]{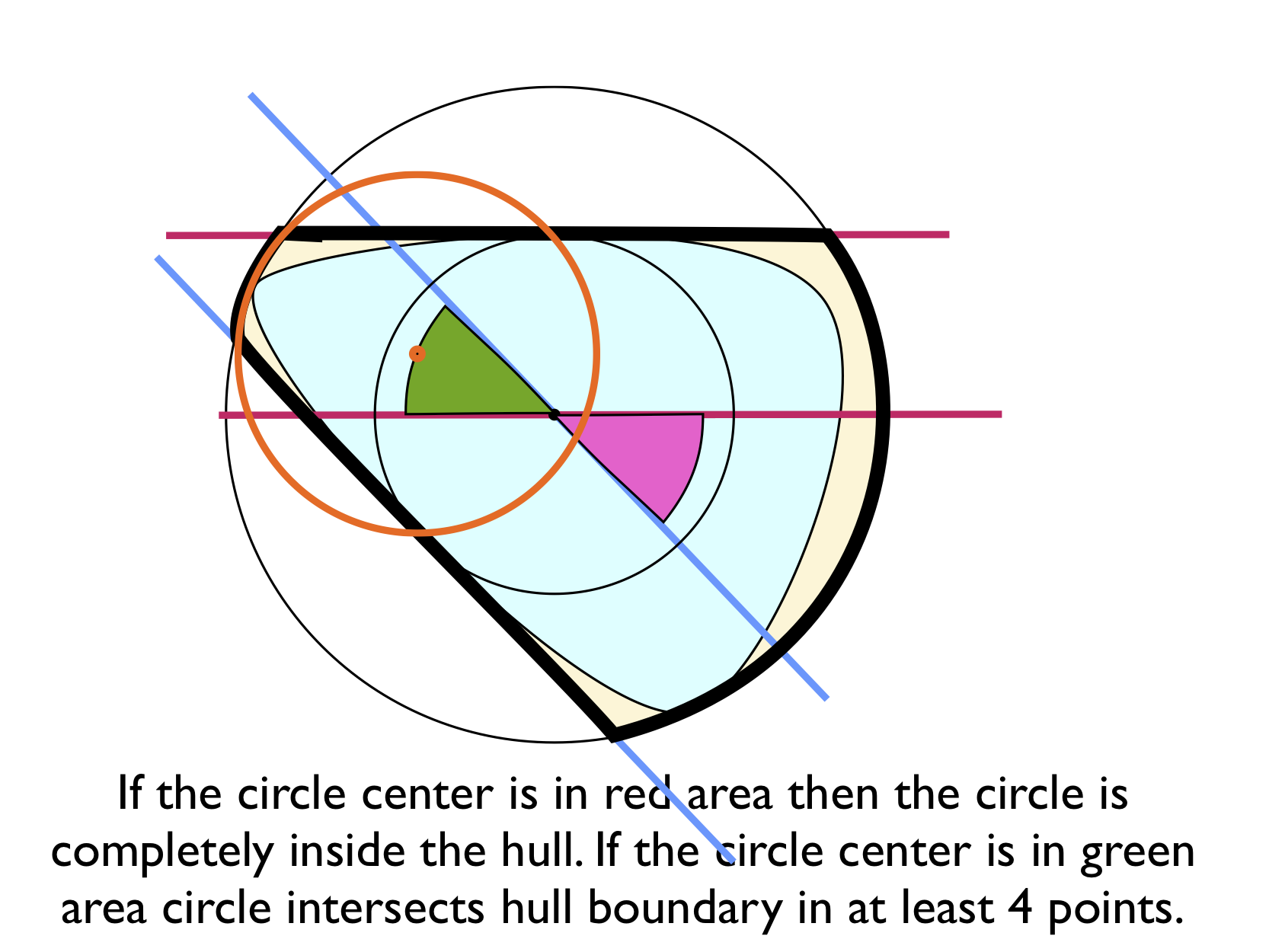}
\caption{averaging trick for the moving circle of radius $\rho_{in}$ showing moving disk}
\label{default}
\end{center}
\end{figure}

 For all positions where the center of the moving disk does not lie in the red region either the boundary of the moving disk intersects the boundary of the fixed lamina or the moving disk and the lamina don't intersect all all.  In either case the integrand $\frac{n}{2}-\nu$ is non-negative. 

Since the average of the integrand is non-negative when integrated over all positions we have $B(\rho_{in})\ge 0$.

Notice that the only facts that we need to draw these pictures and draw conclusions are the inner and outer boundaries of the annulus, their points of contact $a$, $b$, $A$ and $B$ with the convex curve, their tangent lines at $a$ and $b$ and the fact that the boundary curve is convex.   We also use the central symmetry of the circle (or $\mathcal{T}$) in proving the congruence of the red and green regions. 

\newpage
\subsection{$B(\rho_{out})\ge0$}

 \begin{figure}[h]
\begin{center}
\includegraphics[scale=0.5]{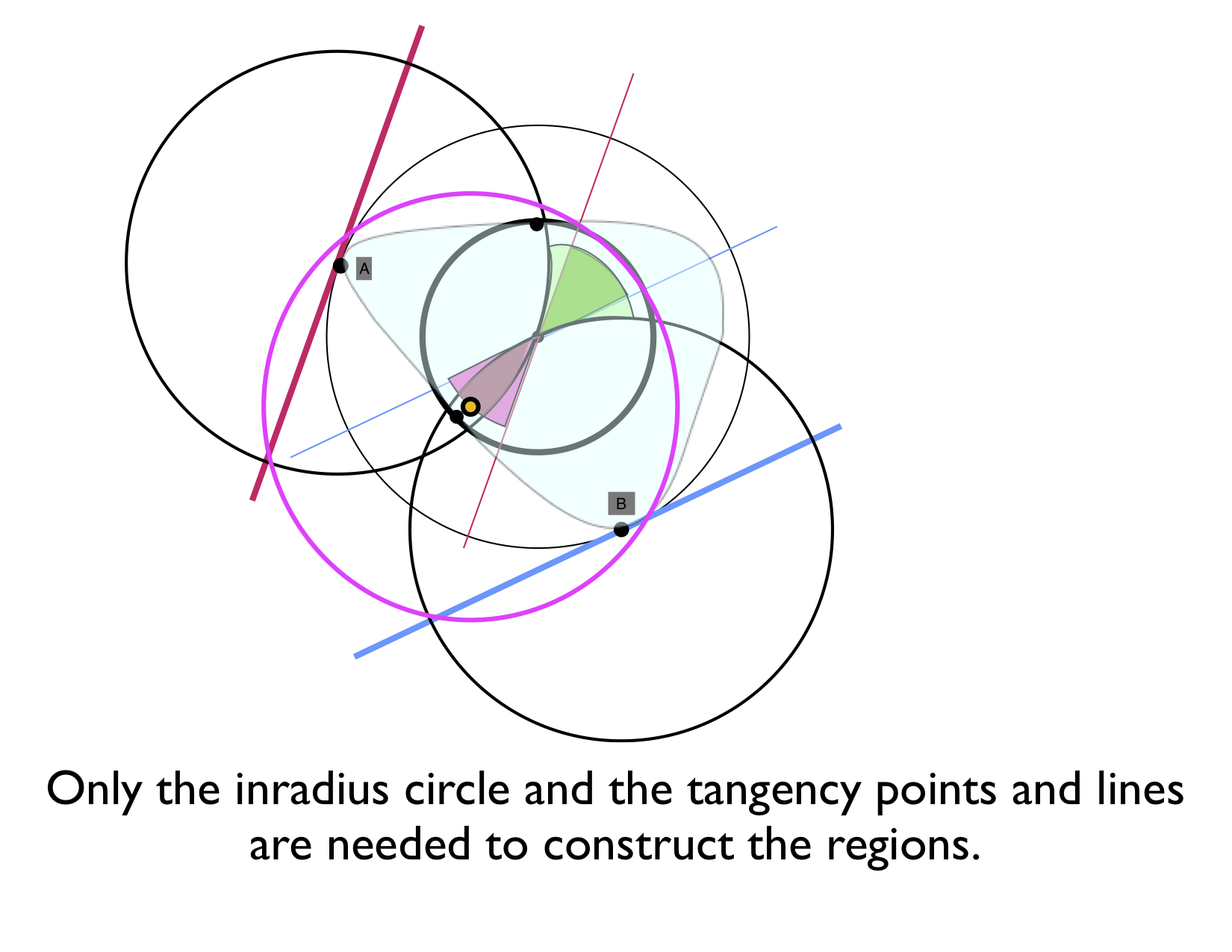}
\caption{averaging trick for the circle of radius $\rho_{out}$}
\label{default}
\end{center}
\end{figure}

The proof that  $B(\rhoout)\ge 0$ is a bit harder to visualize.  

As before we first identify the positions when $n=0$. If the $\rho_{out}$ disk doesn't intersect the lamina at all then $\nu$ is also $0$ which means the integrand is $0$.  If the interior of the lamina and of the disk intersect but the boundaries do not ($n=0$, $\nu=1$)  then the $\rho_{out}$ radius circle completely encloses the lamina and  and the integrand = $-1$).  There are certainly configurations like this since $\rho_{out}\ge \rout$).  We then find balancing regions of equal or greater integral geometric measure where $n\ge4$ and the integrand is greater than or equal to 1. 
 
In order for the moving disk to enclose the entire  fixed lamina it must enclose the inner disk of the annulus and the two points of contact $A$ and $B$ between the outer circle of the annulus and the fixed lamina. Figure 3 illustrates the idea behind this.  We have again drawn the tangents to $A$ and $B$ and their parallel lines through the center of the annulus, the moving disk is represented with a magenta boundary and a yellow center and is positioned so that it encloses $A$ and $B$ and just barely encloses the inner circle of the annulus.  The region in the darker red is the locus of the center of the moving disk where the distance from that center to $A\le \rho_{out}$ and the distance to $B\le \rho_{out}$ as well. These boundaries form curved arcs inside the red triangle.  The bottom arc of the region is the locus of the center when the disk encloses the inner circle of the annulus. In the drawing the boundary of the magenta $\rho_{out}$ radius moving disk still intersects the upper part of the boundary of the lamina in two points, but we could imagine a different lamina, deformed toward the inner circle of the annulus  whose boundary does not intersect the moving disk and whose intersections with the inner and out boundaries of  annulus remain the same.  

Next we find the balancing region:

 \begin{figure}[h]
\begin{center}
\includegraphics[scale=0.5]{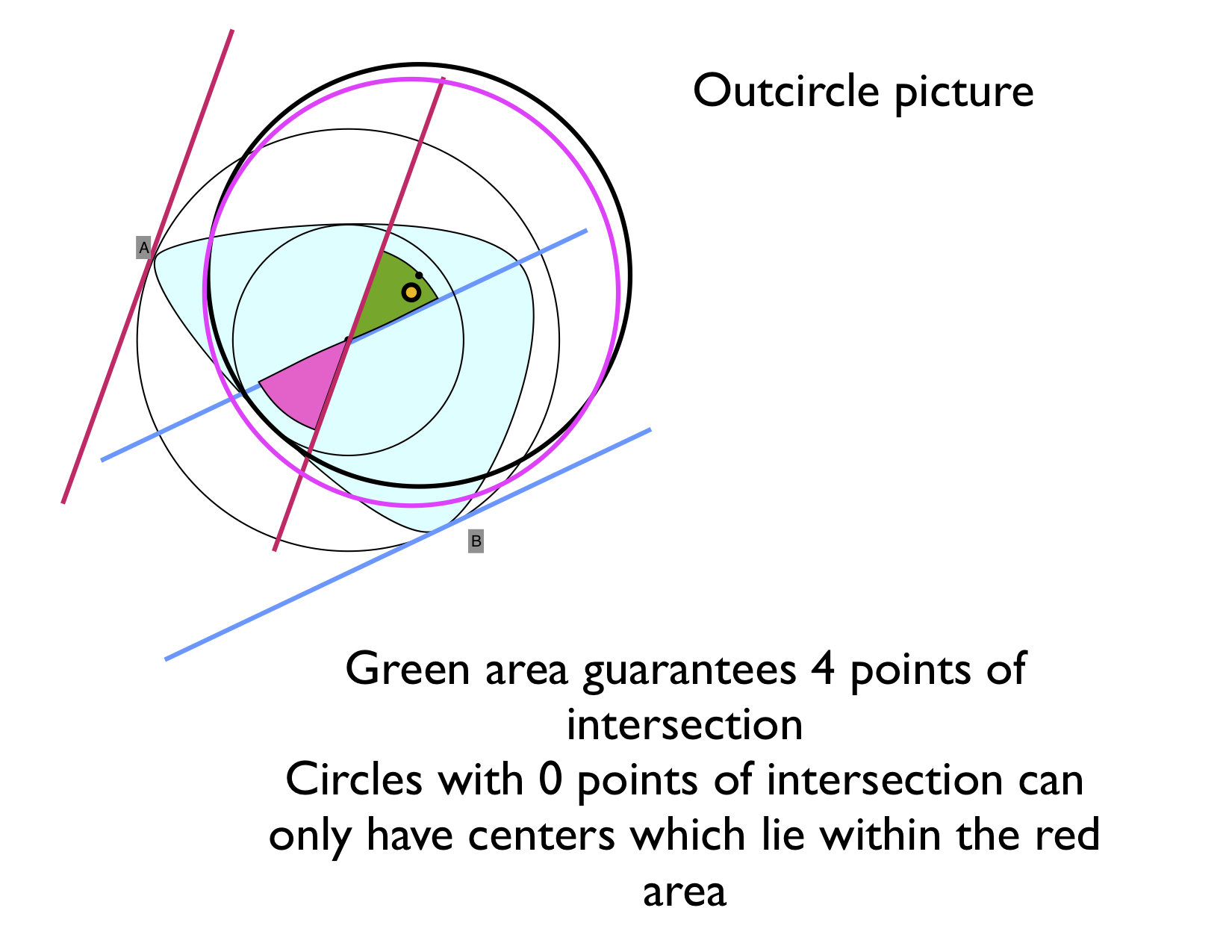}
\caption{averaging trick for the circle of radius $\rho_{out}$}
\label{default}
\end{center}
\end{figure}

In the light green curved triangle of figure 3 the center of the moving $\rho_{out}$ circle is a distance greater than $\rho_{out}$ from $A$ and from $B$.  In figure 4 we show a slightly smaller region on which the center's distance to the tangent lines is greater than $\rho_{out}$ and which is easier to visualize.  The moving circle still encloses the inner annulus so there are arcs of the boundary of the fixed lamina from $A$ to $a$ and to $b$ which must intersect the boundary of the moving circle at least once. Likewise there are two arcs from $B$ to $b$ and to $a$ which intersect the boundary of the moving circle.  Hence $n\ge 4$ and the integrand is on this region is greater than or equal to $1$. 

As in the proof for the $\rho_{in}$ radius circle the two triangles formed by the lines through the center of the annulus are congruent. The red triangle contains the curved region locus of the center where $n$ might be $0$ and the green triangle is contained in the curved region locus of the center where $n$ is guaranteed to be greater than or equal to $4$.  Hence the total integral is non-negative and $B(\rho_{out})\ge 0$.  

This completes the Santalo-Bonnesen integral geometric style proof for Euclidean geometry that $B(h(s))\ge 0$ for all points on the convex curve because $h(s)\in [\rho_{in}, \rho_{out}]$ when the support function $h$ is based at the center of the minimal annulus of the convex curve $\p K$ for Euclidean geometry.
We did not need to use cut, flip and paste techniques as in \cite{Gage1990,Gage1993} to create auxiliary centrally symmetric figures, but we did use the symmetry of the unit disk directly in this proof.  

\newpage
\section{Minkowksi geometry and relative geometry and the weighted curve shortening flow}

In this section of the paper we extend the results of section 1 to the Minkowski geometry/relative geometry case. 

This allows us to prove that the Minkowski version of the Bonnesen functional is non-negative on the interval $[\rho_{in},\rho_{out}]$, from the inner radius of the Minkowski minimal annulus to the outradius. From there one shows that curve shortening in Minkowski geometries also decreases the isoperimetric ratio and that eventually convex curves become ``circular'', that is they converge to the isoperimetrix of the Minkowski geometry.

\subsection{Weighted curvature flow\label{S:weightedcurvature}}

These Minkowski geometries have long been of geometric interest but more recently they   show up independently in a variant of the curve shortening flow where the rate at which the curve shrinks depends on the direction of the normal vector as well as the curvature.  

\begin{equation}\label{E:minkowskicurveshortening}
X_{t}=\gamma(\theta)kN
\end{equation}  

This ``weighted'' curvature equation arises naturally in simplified models of phase boundaries \cite{Taylor1978} and in papers modeling the motion of the boundaries of crystals as they dissolve. The results of \cite{Gage1993} and \cite{Gage-Li1994} show that if $\gamma$ is smooth,  $\pi$ periodic, with strictly positive curvature then any closed, convex curve  evolves to a point under the flow \eqref{E:minkowskicurveshortening} without developing singularities before reaching that point and that further there exists a self-similar flow. This self-similar flow identifies   the  isoperimetrix of the Minkowski geometry associated with  $\gamma$. Notice that it is the isoperimetrix which has self-similar flow, not the unit disk, although for Euclidean geometries these are the same. 

The shrinking shape has support function proportional to $h_{\mathcal{T}}$ and curvature proportional to $k_{\mathcal{T}}$ and it turns out that 	$\gamma = \frac{   h_{\mathcal{T}   }}{   k_{\mathcal{T}}   }$ and describes the ``natural'' curve shortening flow for the Minkowski geometry with isoperimetrix $\mathcal{T}$.

This  equates a physics/material science theory of the weighted curvature flow with the geometric Minkowski curve shortening theory. There is a bijective relation between $\gamma$'s which are smooth, positive and $\pi$ periodic and the convex sets $\mathcal{T}$ symmtric through the origin, which define the isoperimetrix of a symmetric Minkowski geometry.

\begin{itemize}
\item We will describe the Minkovski geometry/  relative convex geometry setting pointing out the analogs and differences with Euclidean geometry.
\item We define the ``right''     curve shortening flow for a symmetric Minkowski geometry and calculate some of its properties. 
\item We define the Minkowski (or relative) version of the Bonnesen functional $\mathcal{B}(r)$ and emphasize  its continued importance in deriving other relative isoperimetric  inequalities  including an integral bound on the Minkowski support function $\frac{h}{\tilde h}$ 
\begin{equation}
\int_{\gamma}\left(\frac{h}{\tilde h}\right)^{2}d\sigma \le \frac{\mathcal{L}A}{\alpha}\label{E:supportsqintmin}
\end{equation}
obtained assuming that $\mathcal{B}(\frac{h}{\tilde h})\ge0$ at each point of $\p K$. This is  the Minkowski analog of the key formula \eqref{E:supportsqint}
used to prove that the  curvature flow decreases the isoperimetric ratio in the Euclidean case. 

See also \cite{Boroczky-Lutwak-Yang-Zhang2012} sections 4 and 5 for  applications to planar cone-volume measure.

(Note: `Relative' and `Minkowski' are used somewhat interchangeably in this paper. `Relative' in this paper means relative to the convex isoperimetrix of the Minkowski geometry which we assume is symmetric through the origin so that it defines a metric. 
\item We will show that the integral geometry proof of the positive center theorem for the minimal annulus can be extended to a proof the relative positive center theorem for the Minkowski minimal annulus.  As in section 1 this proof, using integral geometry, is the main new result of this paper.
\item 
We will align the proofs so that the close correspondence of the Euclidean and Minkowski proofs are apparent. In fact since Euclidean geometry is a special case of a Minkowski geometry (with a circle for the isoperimetrix) the proofs in section 1 become unnecessary and could be replaced by the Minkowski proofs in this section-- although we be believe that it aids the intuition to follow the Euclidean proofs first where your Euclidean intuition is helpful.

\end{itemize}

\subsection{Minkowski geometry background}We first need some background facts about Minkowski geometry.

For plane Minkowski geometry the unit disk $\mathcal{U}$ which defines distance is an arbitrary, centrally symmetric convex body as opposed to the round unit disk.   Area in Minkowski geometry is the same as Euclidean area, but the length of a line segment depends on its direction and is obtained by comparing it to a parallel unit vector in the unit disk. In this paper we will assume that the boundary of the unit disk is smooth since we will be using it to define the coefficients of a differential equation.  This requirement is relaxed in Minkowski geomety and convex relative geometry and can probably be relaxed even in the context of the curve shortening flow.

The isoperimetrix $\mathcal{T}$ is the shape which minimizes the Minkowski length of its boundary for a given area. For Euclidean geometry the isoperimetrix and the unit ball are the same, they are both circles,  but that is not true for most Minkowski geometries. 

A quick summary of the relation between the unit ball and the isoperimetrix:  If the Minkowski unit ball $\mathcal{U}$ is parameterized by $r(\phi)$ in polar coordinates, then the the (Euclidean) support function of the isoperimetrix $\mathcal{T}$ is given by $h_{\mathcal{T}}=\frac{1}{r}$,  often abbreviated $\tilde h(\theta)$ to reduce the number of subscripts.  (This is equivalent to $\mathcal{T}$ being the polar reciprocal of $\mathcal{U}$ rotated by ninety degrees.) see \cite{Gage1993} for a fuller explanation. 

The support function $h(\theta) =h_{K}(\theta)$ is the maximum of the projection of $K$ onto the direction $\vec \theta$ or equivalently the distance to the support line of $\p K$ which is perpendicular to $\vec \theta$. 
This means that since $\p K$ is convex that $h(s)=h(\theta)$ where $\vec \theta$ is perpendicular to the tangent of $\p K$ at $s$. If $\p K$ is strictly convex this is a bijection between $\theta$ and $s$.  The integral geometry results in this section hold even if $\p K$ is flat in some intervals or has corners.

 The closely related theory of relative geometries introduces the concept of mixed volumes of two convex bodies. See section 2 of \cite{Gage1993} for a short but comprehensive  explanation of these geometries and associated inequalities and of the role they play in defining the Minkowski metric and analyzing the weighted curvature flow. 
 
For a closed convex curve the Minkowski length can be defined by  the mixed volume of the interior $K$ of the curve with the isoperimetrix: $V(K,\mathcal{T})$.  For reference we include the formulas for the mixed volumes of $K$ and  $\mathcal{T}$ defined in terms of their Euclidean support functions $h_{\mathcal{T}}$ and $h_K$. 
\begin{align}
\mathcal{L}(\partial K) &=V(K,\mathcal{T}) = \int_{0}^{2\pi} \left(h_{K}h_{\mathcal{T}}-h_{K}^{'}h_{\mathcal{T}}^{'}\right)d\theta = \int \left(h_{K}\tilde h- h'_{K}\tilde h' \right)\,d\theta
\label{E:mixedvolume}\\
2A=\mathcal{A} &= V(K,K)= \int_{0}^{2\pi} \left( h_K^{2} -(h_K')^2\right)d\theta \text{  ( i.e. the mixed volume is twice the Euclidean area of $K$})\\
2\alpha& = V(\mathcal{T},\mathcal{T})= \int_{0}^{2\pi}(\tilde h^{2} -\tilde h'^{2})\, d\theta    \text{  ( twice the Euclidean area  of the isoperimetrix)}
\end{align}

Note: In \cite{Gage1993} the normalization was to take $\alpha$ to be $V(\mathcal{T}, \mathcal{T})$ as is commonly done in relative geometry papers, i.e. twice what it is in this paper.  Our choice here emphasizes the compatibility with Euclidean forumula, for example if the isoperimetric is a Euclidean unit circle then $\alpha$ equals $\pi$, the area of the unit disk, rather than $2\pi$.  One needs to be cautious with these factors of 2 as one reads relative geometry papers since different authors have used different conventions over the years.

Integrating  the first equation by parts we obtain an equivalent definition of the Minkowski length in terms of the unit disk. 
\begin{equation}\label{mixedvolume}
\mathcal{L}(\p K) =\int_{\p K}  h_{K}(\tilde h +\tilde h '') d\theta =\int \tilde h(h_{K}+h_{K}'' )d\theta = \int \tilde h\frac{d\theta}{k_{K}}= \int \tilde h ds_{K}
\end{equation}
We can also relate this to Euclidean arclength and Minkowski arclength as
\begin{equation}\label{identities}
\mathcal{L}(\partial K)=\int \tilde h ds_{K}=\int \frac{	ds_{K}}{r}=\int d\sigma 
\end{equation}
since $d\theta = k_{K}ds_{K}$ where $k_{K}$ is the euclidean curvature of the boundary of $K$ and $ds_{K}$ is the euclidean element of arc length
Similarly  $d\sigma = \tilde h ds_{K}$ is the Minkowski element of arc length .

\subsection{Basic Minkowski identies\label{S:basicidentiesminn}}\hfill

We summarize these identities which appear in equations \eqref{mixedvolume} and \eqref{identities} for use later.
\begin{itemize}
\item $h$ or $h_{K}$ is the support function of the convex curve
\item $h_{\mathcal{T}}$	 or $\tilde h$ is the Eucidean support function of the isoperimetrix
\item $k_{K}$ or $k$ is the Eucidean curvature of the curve
\item $k_{\mathcal{T}}$ or $\tilde k$ is the Eucidean curvature of the isoperimetrix
\item $ds = \frac{d\theta}{k}$ is Euclidean arclenth of the curve
\item $d\sigma = \tilde h ds = \frac{ds}{r}$ is the Minkowski arc length where $r$ is the radial function for the unit disk $\mathcal{U}$
\end{itemize} 
These will be used to simplify the notation for some of the formulas below to align them with the more familiar formulas of Euclidean geometry where $r$ and $\tilde h$ are identically $1$.

\subsection{The centrality of the Minkowski Bonnesen functional}
The Bonnesen inequality is equally central in proving isoperimetric inequalities for Minkowski geometries.
\begin{itemize}
\item We define the Minkowski Bonnesen functional for a smooth closed curve $\p K$ bounding a lamina $K$ in the plane as
\begin{equation}
\mathcal{B}(r) :=r\mathcal{L}-A-\alpha r^2
\end{equation}
\item If $\mathcal{B}(r)\ge0$ for even one value $r$ then the quadratic functional has real roots and its discriminant $\mathcal{L}^{2}-4\alpha A$ is greater than or equal to $0$.  This is the isoperimetric inequality.   
\item If $\mathcal{B}(\rin)\ge 0$, where $\rin$ is the inradius, the value of the largest $r$ such that $r\mathcal{T}$ is enclosed by the curve and also $\mathcal{B}(\rout)\ge0$ where $\rout$ is the out radius, the value of the smallest $r$ such that $r\mathcal{T}$	 encloses the curve, then, because the $B(r)$ quadratic is concave down, $\mathcal{B}(r)\ge 0$ for all $r$ in the interval $[\rin, \rout]$.
\item The Minkowski geometry version of the Bonnesen inequality follows from this as in the introduction \ref{T:bonnesenInequalityIntro}.
\begin{equation}
\mathcal{L}^{2}- 4\alpha A \ge\alpha^{2}(\rout-\rin)^2
\end{equation}
\item As in the first section we will define the minimal width Minkowski annulus and using integral geometry methods and the ``averaging'' trick we'll show that $\mathcal{B}(r)\ge 0$ for $r\in[\rhoin,\rhoout]$ which implies
$$ \mathcal{L}^{2}- 4\alpha A \ge\alpha^{2}(\rhoout-\rhoin)^2$$
\end{itemize}

\subsection{The curve shortening flow associated with a Minkowski geometry}

In Euclidean geometry the circle shrinks in a self-similar fashion under the Euclidean curve shortening flow.  It turns out as mentioned in section \ref{S:weightedcurvature} that the appropriate Minkowski analog is a flow for which the Minkowski isoperimetrix shrinks self-similarly.  The weighted curvature flow where
\begin{equation}\label{rightgamma}
\gamma(\theta)= \frac{h_{\mathcal{T}}}{k_{\mathcal{T}}}=\frac{\tilde h}{\tilde k}
\end{equation}

in \eqref{E:minkowskicurveshortening} does exactly that. Once this ``right'' weighted curvature flow is chosen the main result of \cite{Gage1993} uses the Bonnesen inequality to show that  with this flow, any closed convex evolves to a point and converges asymptotically  to a shrinking $\mathcal{T}$.  This also implies that the self-similar flow for the equation is unique.

\begin{theorem}
For smooth, strictly convex curves, the evolution equation \eqref{E:minkowskicurveshortening} with 	$\gamma =\frac{\tilde h}{\tilde k}$ can be rewritten as the evolution equation for the support function
\begin{equation}\label{E:evolsupport}
h_{t}(\theta)=-\frac{\tilde h}{\tilde k}k=-\tilde h(\tilde h +\tilde h'')k
\end{equation}
\end{theorem}
\begin{proof}
This amounts to adding a tangential component to the original flow, which reparametrizes the curve according by the angle of the normal vector, but does not change its shape. (See \cite{Epstein-Gage1987})
\end{proof}
 Note that if $\tilde h\equiv 1$ (the Euclidean case) then equation \eqref{E:evolsupport} becomes
 $h_{t}=-k$ which is clearly satisfied by the shrinking circle.  

\begin{lemma}

A self-similar solution to equation \eqref{E:evolsupport} is $h(t,\theta)=\tilde h (\theta)\sqrt{2T-2t}$ where $T$ is the time when the isoperimetrix shrinks to a point. Conversely if $h(t,\theta)=\lambda(t)h(0,\theta)$ is a self-similar solution of \eqref{E:evolsupport} then $h(0,\theta)/k(0,\theta) $ is proportional to $\gamma(\theta)$.
\end{lemma}

This verifies that this choice for $\gamma$ is the ``right'' curve shortening flow, the flow for which the isoperimetrix evolves self-similarly. 	

\subsection{Derivatives of Minkowski length and area under the curve shortening flow}
\begin{lemma}	
$$2A_{t}= \mathcal{A}_{t}=-2\alpha = -2\int \gamma d\theta$$ 
\end{lemma}
\begin{proof}
The proof is a straightforward calculation when done in terms of the support function. Recall $h +h'' =\frac{1}{k}$
\begin{align}
2A_{t} &= \mathcal{A}_{t}=2\int \left( h h_{t}-h'h'_{t}\right) \, d\theta\\
 &= 2\int\frac{\tilde h _{t}}{k}d\theta\notag\\
&= -2\int \frac{\tilde h}{\tilde k}	d\theta = -2\int \gamma d\theta\notag\\
&= -2\int \left(\tilde h^{2}- \tilde h'^{2}\right) d\theta = -2\alpha
\end{align} 
\end{proof}
\begin{lemma}{}
$$ \mathcal{L}_{t}= -\int \left(\frac{k}{\tilde k}\right)^2d\sigma $$ 
\end{lemma}
\begin{proof}
The calculation is immediate in terms of the support funcition
$$\mathcal{L}_{t} =\int \frac{h_{t}}{\tilde k}d\theta = -\int \frac{\tilde h k}{\tilde k^2}d \theta =\int \tilde h \frac{k^{2}}{\tilde k^{2}}ds = \int \left(\frac{k}{\tilde k}\right)^{2}d\sigma 
$$

\end{proof}

\subsection{The derivative of the Minkowski isoperimetric ratio} 
\begin{lemma}			
\begin{equation}\label{E:derivativeofIP}
\left(\frac{\mathcal{L}^{2}}{A}\right)_{t} = 
= -2\frac{\mathcal{L}}{A}
\left( \int\left( \frac{k}{\tilde k}\right)^2d\sigma - \alpha \frac{\mathcal{L}}{A}\right) 
\end{equation}
\end{lemma}
\begin{proof}
Immediate calculation proves the lemma.
\end{proof}

While the curvature, length and area do not depend on the choice of an origin the support functions $h$ and $\tilde h$ do depend on the origin.

\begin{lemma}Suppose that there is a point $\mathcal{O}$ such that support functions based at $\mathcal{O}$ satisfy
\begin{equation}\label{E:supportsqintminkowski1}
\int \left(\frac{h}{\tilde h}\right)^2d\sigma \le \frac{\mathcal{L}A}{\alpha} 
\end{equation}
then 

$$\left(\frac{\mathcal{L}^{2}}{A}\right)_{t}\le0$$

and the isoperimetric ratio is non-increasing.
\end{lemma}
\begin{proof}
\begin{align}
	\mathcal{L}^{2} &= \left( \int \frac{h}{\tilde k}d\theta\right)^{2} =\left( \int \frac{h}{\sqrt{k\tilde h}}\frac{\sqrt{k\tilde h}}{\tilde k}d\theta\right)^{2}\\
&=\int  \frac{h^{2}}{\tilde h k} d\theta \int \frac{k\tilde h}{\tilde k^{2}}d\theta
= \int \frac{h^{2}}{{\tilde h}^{2}}d\sigma \int \frac{k^{2}}{\tilde{k}^{2}}d\sigma
\end{align}
using the inequality

$\int \left(\frac{h}{\tilde h}\right)^2d\sigma \le \frac{\mathcal{L}A}{\alpha} $ on the right hand side and simplifying we obtain

$
\int \frac{k^{2}}{\tilde{k}^{2}}d\sigma\ge \alpha \frac{\mathcal{L}}{A}
$
which when plugged in to  equation \eqref{E:derivativeofIP} proves that the isoperimetric ratio is non-increasing. 
\end{proof}

We have reduced the problem of showing that the isoperimetric ratio is non-increasing to finding a point $\mathcal{O}$   for basing the support functions where \eqref{E:supportsqintminkowski1} is satisfied.
\subsection{Relative positive center}

\begin{definition}	A \textbf{relative positive center} $\mathcal{O}$  of a convex boundary curve $\p K$ is a point where the relative support function based at $\mathcal{O}$ satisfies $\mathcal{B}\left(\frac{h(\theta(\sigma))}{\tilde h(\theta(\sigma))}\right)\ge 0$ for each point $\sigma$ on the boundary curve $\p K$.
\end{definition}

\begin{theorem}\label{T:supportsqintminkowski2}If  a relative positive center $\mathcal{O}$ exists for $\p K$ then
the relative support function $\left(\frac{h}{\tilde h}\right)$ based at $\mathcal{O}$ satisfies
\begin{equation}
\int \left(\frac{h}{\tilde h}\right)^2d\sigma \le \frac{\mathcal{L}A}{\alpha}\notag
\end{equation} 
\end{theorem}
\begin{proof}
\begin{align}
0&\le \int_{\p K} B\left(\frac{h}{\tilde h}\right )\frac{\tilde h}{k}d\theta\\ \notag
&=\int_{\p K}\frac{h}{\tilde h}\frac{\tilde h}{k}\mathcal{L}-\frac{\tilde h}{k}A-\left(\frac{h}{\tilde h}\right)^{2}\frac{\tilde h}{k}\alpha d\theta\\ \notag
&\int_{\p K}\mathcal{L}\frac{h}{k}-A\frac{\tilde h}{k}-\frac{h^2}{\tilde h k}\alpha d\theta\\ \notag
&= \mathcal{L}\int  h ds -A \int \tilde h ds -\frac{h^2}{\tilde h k}\alpha d\theta\\ \notag
&= 2\mathcal{L}A-A\mathcal{L}- \frac{h^2}{\tilde h k}\alpha d\theta\\ 
&= \mathcal{L}A -\alpha \int \left(\frac{h}{\tilde h}\right)^2 d\sigma
\end{align}
\end{proof}

So now showing that the isoperimetric ratio is non-increasing has been reduced to  showing that a relative positive center  exists.

\subsection{The symmetric curve case}
If the curve $\p K$ is symmetric through the origin the existence of a relative positive center is obvious.  Choose $\mathcal{O}$ to be the origin then $\p K$, inradius isoperimetrix $\rin \mathcal{T}$ and the outradius isoperimetrix $\rout \mathcal{T}$ are all centered at the origin and $\rin \mathcal{T}\subseteq K \subseteq \rout \mathcal{T}$ form a minimal annulus around $\p K$.  Hence $\rin\le\frac{h}{\tilde h}\le\rout$ and the original Santalo-Bonnesen proof shows that $\mathcal{B}(\frac{h}{\tilde h})\ge 0$ for all points on $\p K$. (See Theorem \ref{T:MinkowskiBonnesenInequality} in the next section.)

We have shown via theorem \ref{T:supportsqintminkowski2} and equation \ref{E:derivativeofIP} that  the Minkowski isoperimetric ratio for a symmetric curve is non-increasing under the   weighted curvature flow which corresponds to the isoperimetrix. Using results from \cite{Gage1984} and \cite{Angenent1991} one can show that the flow shrinks the curve to a point and unless the curve is the boundary of the isoperimetrix, the isoperimetric ratio is strictly decreasing and approaches $4\alpha$ as the area of $K$ approaches $0$


Our next task  is to show that there is an integral geometry approach similar to the Euclidean geometry that we introduced in section 1 of this paper which will work for Minkowski geometries.  The Poincare and Blaschke formulas for Minkowski geometries are not as well known as their Euclidean versions but once the definitions and formulas of Minkowski integral geometry are properly aligned with the Euclidean terminolgy the proofs are very similar.

\subsection{Integral geometric measure for Minkowski geometry} 
Since rotation is not a Minkowski geometry isometry in most cases , the appropriate invariant integral geometric measure is dxdy, without the rotation variable $d\theta$. It is called ``translative integral geometry'' by Schneider, (p254 \cite{Schneider1993}) but was used by
many previous authors including \cite{Chakerian1962} and \cite{Miles1974} -- supplemented by \cite{Firey1977} and is also described in footnotes in \cite{Santalo1976}

\subsection{Crofton formula for Minkowski geometry}
Sections 1 and 2 of \cite{Chakerian1962} give an excellent introduction to the Minkowski unit disk $\mathcal{U}$ (called the Indicatrix by Chakerian), the dual isoperimetrix $\mathcal{T} $and the relation between them.  He also gives a density for lines and a Crofton formula which we will not need for this proof.  See also \cite{Gage1994} for a similar discussion in the context of Minkowski curvature flow.

\subsection{Poincare formula for Minkowski geometry}
We do need the  Poincare formula for the boundary intersection of two convex bodies $K_{0}$ and $K_{1}$.

A reference from \cite{Santalo1976} (footnotes on p124) gives
``a poincare formula when only translations are taken into account'' which says that 
\begin{equation}\label{E:dxdypoincare}
\int_{\partial K_{0}\cap\partial K_{1}} n dxdy = 2\int_{\partial K_{0}\cap\partial K_{1}}[h_{1}(\theta)+h_{1}(\theta+\pi)]ds_{0}=4(F_{01}+F_{01}^{*})
\end{equation}
where in Santalo's notation $F_{01}$ is  \textbf{half} the mixed area of $K_{0}$ and $K_{1}$  and $F_{01}^{*}$ is half the mixed area of $K_1$ with $K_{0}^*$  which is $K_0$ reflected  through the origin.  
Taking $K_{0}=r\mathcal{T}$, noting that $\mathcal{T}$ remains unchanged when reflected through the origin and applying the formula \eqref{E:mixedvolume} for Minkowski length  
 we obtain 
\begin{equation}\label{E:dxdypoincare2}
\int n dxdy =4r\mathcal{L}
\end{equation}

\subsection{Blaschke formula a.k.a the kinematic formula}

See \cite{Santalo1976} (p124 footnotes)

\begin{equation}\label{E:dxdyblaschke}
 \int _{K_{0}\cap K_{1}\ne\emptyset}  dP_{1}=F_{0}+F_{1}+2F_{01}
 \end{equation}

where all of the $F$'s are \textbf{half} mixed areas, e.g. $F_{0}$ and $F_{1}$ are the  areas of $K_{0}$
and $K_{1}$ while $F_{01}=\frac12 V(K_{0} ,K_{1})$.   (\cite{Santalo1976} (see p 4 for definitions and notations). 
Taking $K_0=r\mathcal{T}$, using our notation and observing that $\nu$, the number of intersections of the interiors, is either $1$ or $0$ because these are convex sets, 

\eqref{E:dxdyblaschke} becomes
\begin{equation}\label{E:dxdyblaschke2}
 \int _{r\mathcal{T}\cap K_{1}\ne\emptyset} \nu dP_{1}=r^{2}\alpha+A+r\mathcal{L}
\end{equation}

See also \cite{Miles1974} and \cite{Firey1977}.

\subsection{Determining the interval on which the Minkowski Bonnesen functions $\mathcal{B}$ is positive}\hfill

We define the Minkowski version of the Bonnesen functional as 
\begin{equation}\label{E:dxdyBonnesenfunctional}
\mathcal{B}(r)=r\mathcal{L}-A- r^{2}\alpha
\end{equation}

and by subtracting \eqref{E:dxdyblaschke2} from \eqref{E:dxdypoincare2}and  we see that the functional $B$ can be represented by the integral geometric formula
\begin{equation}\tag{Minkowski Bonnesen functional}
\int_{r\mathcal{T} \cap K_{1}\ne\emptyset} (\frac{n}{2}- \nu)dxdy = \mathcal{B}(r)=r\mathcal{L}-A-r^{2}\alpha \label{E:dxdybonneseneq}
\end{equation} 
 \begin{proof} of formula \eqref{E:dxdybonneseneq}
 \begin{align}
\int_{r\mathcal{T}\cap K_{1}\ne\emptyset} \frac{n}{2} dxdy 
=& 2r\mathcal{L}\\
\int_{r\mathcal{T}\cap K_{1}\ne\emptyset} \nu dxdy
=& r^{2}\alpha +A + r\mathcal{L} 
\end{align}
\end{proof}

This is nearly identical to the integral geometric representation of the Bonnesen functional in the first section of this paper.  For Euclidean geometry $\alpha$ becomes $\pi$, the area of the unit disk.

\begin{theorem} [The Minkowski Bonnesen inequality]
\label {T:MinkowskiBonnesenInequality}
$\mathcal{B}(r)\ge 0$ for all $r\in [\rin,\rout]$

As explained in the introduction \ref{T:bonnesenInequalityIntro} it follows easily from this that $\mathcal{L}^{2}-4\alpha A \ge \alpha^{2}(\rout - \rin)^{2}$.
\end{theorem}

\begin{proof}
 If the moving isoperimetrix $r\mathcal{T}$ has radius $r$ with $r\in[\rin,\rout]$ then the isoperimetrix can neither contain the lamina $K$ bounded by the curve nor be contained in that lamina.  Hence, except for positions of measure $0$, every time the isoperimetrix touches the lamina ($\nu=1$) the boundary of the isoperimetrix must intersect the boundary of the lamina (generically) in two or more places.  This implies that the integrand is always non-negative and that $B(r)\ge0$ for $r\in[\rin,\rout]$.
 This is Santalo's proof that the Minkowski Bonnesen functional $B(r)$ is greater than or equal to $0$  when $r\in  [r_{in},r_{out}]$ and that equality occurs only when   $K$. is the isoperimetrix.
 \end{proof}
 
 Our proof of the strengthened version of this inequality will show that $\mathcal{B}(r)\ge 0$ for $r\in[\rhoin,\rhoout]$, the range between the inner radius of the minimal annulus and the outer radius of the annulus, even though the integrand is sometimes negative, and that therefore the center of the relative minimal annulus is a relative positive center, that is $\mathcal{B}\left(\frac{h(\theta(\sigma))}{\tilde h(\theta(\sigma))}\right)\ge 0$ for all $\sigma \in [\rhoin,\rhoout]$.
 
 It does this by balancing the measure of positions where the isoperimetrix $\rhoin \mathcal{T}$ is contained within  $K$,  which have intersection numbers $n=0$ and $\nu=1$ and the integrand is negative, with positions  where $n\ge 4$ and $\nu=1$ and the integrand is positive. 
This implies that the integral is positve even though the integrand is sometimes negative and therefore $\mathcal{B}(\rhoin)\ge0$.  To the best of my knowledge this balancing approach has not previously been used in this context.  We then do a similar analysis with the isoperimetrix $\rhoout \mathcal{T}$ to prove that $\mathcal{B}(\rhoout)\ge0$. Since $\mathcal{B}$ is convex downard this will complete the proof that $\mathcal{B}\ge 0$ on $[\rhoin,\rhoout]$ in the Minkowski geometry case. 

Before we do that we need a theorem that a Minkowski minimal annulus always exists and again we refer to results in \cite{Gage1993}. 

\begin{theorem}[Minkowski minimal annulus theorem]\label{T:minkowskiminimalannulus} let $K$ be the lamina interior to the convex curve and let $\mathcal{T}$ be the isoperimetrix. We can find a translation and  homotheties of $\mathcal{T}$  so that   
 $$\rho_{in} \mathcal{T} \subseteq K \subseteq \rho_{out} \mathcal{T}$$
  and $\rho_{out}-\rho_{in}$ is minimized and we can prove that in that  case and there are points $a, A, b, B$ on the boundary of $K$ which alternately touch tangentially $\rho_{in}\mathcal{T} $
 ( at $a$ and $b$) and $\rho_{out}\mathcal{T}$ (at $A$ and $B$).  In addition $\rhoin\le \frac{h}{\tilde h} \le \rhoout$.
 \end{theorem}
 Notice that this is identical with the minimal annulus theorem for Euclidean geometry except for using the unit isoperimetrix instead of the unit Euclidean disk. 
\begin{proof}
In  \cite{Gage1993} this theorem is proved by using   approximation theory  applied to the support function. 
It is clear that  the existence of a minimal Euclidean annulus is equivalent to the $L_{0}$ approximation problem of  finding $a$ and $b$ so that the range of 
 $$\frac{h(\theta)}{\tilde h(\theta)}+ \frac{a\cos \theta}{\tilde h(\theta)}+ \frac{b\sin \theta}{\tilde h(\theta)}$$
 is minimized. 
 
 See \cite{Gage1993} Lemma 5.2 for details for completing the proof along with references to G. G. Lorentz's book ``Approximation of Functions'' \cite{Lorentz1966} particularly pages 18-27 on Cheyshev systems of functions. ( Bonnesen's work included a lot of approximation theory and he was probably aware of this connection with geometry.  )  The proof in \cite{Gage1993} actually proves the Minkowski case and deduces the Euclidean case by specifying that the isoperimetrix is a Euclidean disk and that $\tilde h \equiv 1$. 
\end{proof}

 \subsection{Relative positive center}

\begin{theorem}\label{T:relpositivecenter} Every convex curve $\p K$ has at least one relative positive center: the center of the relative minimal annulus enclosing $\p K$.  We will show that $\mathcal{B}(r)\ge0$ for $r\in[\rhoin,\rhoout]$.  And since $\rhoin \le \frac{h}{\tilde h}\le\rhoout$  it follows that 	$\mathcal{B}\left(\frac{h}{\tilde h}\right)\ge 0$ everywhere on $\p K$. The ratio $\frac{h}{\tilde h}$ is called the relative support function.
\end{theorem}

We will  prove these lemmas using the  integral geometric representation of the Bonnesen functional  \eqref{E:dxdybonneseneq} above)   and the balancing arguments used for the Euclidean plane. Because of the convexity of the Bonnesen functional we only need to prove that $\mathcal{B}(\rhoin)\ge0$ and that $\mathcal{B}(\rhoout)\ge0$ which is done in the subsequent lemmas.

 The only facts used in the Euclidean balancing proof are the convexity of the objects, the inner and outer boundaries of the annulus, the four points of contact and their tangent lines. The actual shape of the annulus boundary or the lamina boundary are needed only to help with the intuition.  Hence the wording for the proof in the Minkowski case is the same as for the Euclidean case, only the images are different ---  boundaries of convex bodies instead of circles.  
 
  \newpage
 In the next two figures the inner and outer boundary of the relative minimal annulus are shown along with the four contact points and their tangent lines, but the lamina itself is not drawn.  The lines through the center of the annulus parallel to the tangent lines are drawn as is the curve $\partial \left((\rhoout-\rhoin) \mathcal{T}\right)$ outside of which the integrand is guaranteed to be non-negative, either because $n\ge2$ or $n=\nu=0$. ( Compare with the proofs in the Euclidean case where the lamina is drawn to add additional intuition.)

 \begin{lemma}
 {$\mathcal{B}(\rho_{in)}\ge0$)}
 \end{lemma}
 \begin{proof}
\begin{figure}[h]
\begin{center}
\includegraphics[scale=0.5]{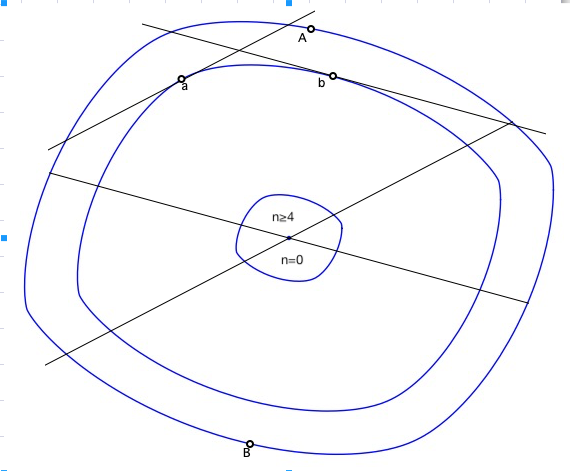}
\caption{averaging trick for the figure of radius $\rho_{in}$}
\label{default}
\end{center}
\end{figure}
 
We are considering the locus of the center of a moving Minkowski disk congruent to $\rhoin\mathcal{T}$ (this would be the center of a circle of radius $\rhoin$ in the Euclidean case).  The only configurations of interest are ones where the disk intersects the fixed lamina ($\nu=1$) - in all other cases $\nu=n=0$ and the integrand is zero.  As in the Euclidean case we need to insure that the area of the locus of the center when $n\ge4$ is larger than the area of the locus where $n=0$.

Considering Fig1 and Fig 5 we see that the situation is very similar, since both the Euclidean and Minkowski disk are convex sets, and we have the 4 points $a, A, b, B$ where the inner and outer disks of making up the annulus are tangent to the fixed lamina.
These in turn create two lines through the origin parallel to the tangent lines. 

In one of the sectors created by these lines through the origin the centers are such that $\rho_{in}\mathcal{T}$ will touch both the tangent lines and provide 4 points of intersection. As in the Euclidean case you can verify that these points are distinct, except for  exceptional points where the center is on the boundary of the inner body and the moving disk touches the outer boundary of the annulus (at $A$ and perhaps other points). This region is lower dimensional and won't contribute to the integral. 

The region where $n=0$ is symmetric through the origin to the region where $n\ge4$ so the net contribution to the integral is 
non-negative just as in the Euclidean case.  Here we are using the fact that the isoperimetrix is symmetric through the origin to insure that the ``caps'' on the two regions are the congruent. 
\end{proof}
\newpage 
 \begin{lemma}{$\mathcal{B}(\rho_{out)}\ge0$)}
\end{lemma}
\begin{proof}
 \begin{figure}[h]
\begin{center}
\includegraphics[scale=0.5]{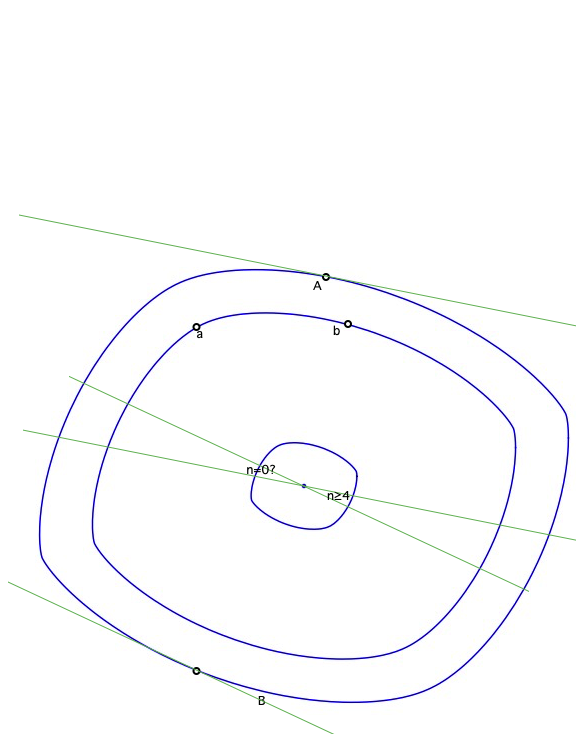}
\caption{averaging trick for the figure of radius $\rho_{out}$}
\label{default}
\end{center}
\end{figure}

The story is  similar for determining the intersections of the moving $\rho_{out}\mathcal{T}$ ``disk''.  In this case the region where $n$ might be $0$ is the one containing centers of moving disks that might include the entire lamina and would therefore include $A$ and $B$.  These centers, as in the Euclidean case, need to lie the sector where their isoperimetrix  includes both $A$ and $B$ and  therefore the center must lie in the sector where its isoperimetrix  touches the tangent lines through $A$ and $B$. 

The region congruent through the origin is one containing centers where the disk can't reach the outer tangent line and therefore can't completely contain the fixed lamina.  Since the moving disk still encloses the inner annulus there are arcs of the boundary of the fixed lamina from $A$ to $a$ and to $b$ which must intersect the boundary of the moving isoperimetrix at least once. Likewise there are two arcs from $B$ to $b$ and to $a$ which intersect the boundary of the moving circle.  Hence $n\ge 4$ and the integrand is on this region is greater than or equal to $1$. 

We also consider positions where the center of the moving disk lies outside the circle of radius $\rho_{out}-\rho_{in}$.  These disks must have points which lie outside the outer boundary of the annulus so they can't be contained within the lamina. If they don't intersect the lamina at all then $n=0$ and $\nu=0$ and the integrand is zero and their contribution to the integral is $0$.  If they do intersect the lamina   then the boundaries of the lamina and the disk intersect and the integrand $\frac{n}{2}-\nu$ is at least $0$ and so they contribute non-negatively to the integral. 

 Once again the argument is identical to the Euclidean case since both the Euclidean disk and the Minkowski disk are convex bodies.    
\end{proof}

\begin{proof}  of theorem \ref{T:relpositivecenter}

Having shown that $B(\rhoin)$ and $B(\rhoout)$ are both non-negative it is follows that $B(r)\ge 0$ for $ r\in[\rhoin,\rhoout]$ hence $B(\frac{h}{\tilde h})\ge0$ on $\p K$ which means that the center of the relative minimum annulus is a positive center.
\end{proof}

Observe that  once the Minkowski geometry terminolgy has been defined in a way that aligns with the Euclidean geometry terms,  the proofs  for Minkowski geometry follow the original Euclidean proofs very closely.

The results of this theorem, but not the method of proof, are   the same as theorem 5.1 of \cite{Gage1993} where the cut, flip, paste trick was used. The theorem also appears as theorem 1 in \cite{Peri-Wills-Zucco1993} using the cut, flip, paste trick.

 The proof using the Santalo integral geometry method and balancing regions where  the integrand is negative with regions where the integrand is positive is, I believe,   new.

 \end{document}